\documentclass{amsart}
\usepackage{amssymb}
\usepackage{amsfonts}

\newtheorem{theorem}{Theorem}[section]
\newtheorem{corollary}[theorem]{Corollary}
\newtheorem{lemma}[theorem]{Lemma}
\newtheorem{definition}[theorem]{Definition}
\newtheorem{proposition}[theorem]{Proposition}
\newtheorem{example}[theorem]{Example}
\newtheorem{remark}[theorem]{Remark}
\numberwithin{theorem}{section}

\input{tcilatex}

\begin{document}
\title[Completely positive maps on Hilbert $C^{\ast }$-modules]{Comparison
of completely positive maps on Hilbert $C^{\ast }$-modules}
\author{Maria Joi\c{t}a}
\address{Department of Mathematics \\
University of Bucharest\\
Bd. Regina Elisabeta nr. 4-12\\
Bucharest, Romania}
\email{ mjoita@fmi.unibuc.ro}
\urladdr{http://sites.google.com/a/g.unibuc.ro/maria-joita/}
\subjclass[2000]{Primary 46L08, 46L05}
\keywords{completely positive maps, Hilbert $C^{\ast }$-modules, Stinespring
construction}

\begin{abstract}
We introduce a preorder relation in the collection of all operator valued
completely positive maps on a full Hilbert $C^{\ast }$-module and
characterize this relation in terms of the Stinespring construction
associated to each completely positive map.
\end{abstract}

\maketitle

\section{Introduction and preliminaries}

The study of completely positive maps is motivated by their applications to
quantum information theory, where operator valued completely positive maps
on $C^{\ast }$-algebras are used as a mathematical model for quantum
operations, and quantum probability. Many problems from quantum information
theory involve characterization and comparison of quantum operations. The
structure theorems and the Radon-Nikodym type theorems for completely
positive maps between $C^{\ast }$-algebras play an important role to
characterization and comparison of quantum operations and so to
understanding of certain problems from quantum information theory.
Stinespring \cite[Theorem 1]{S} showed that an operator valued completely
positive map $\varphi $ on a unital $C^{\ast }$-algebra $A$ is of the form $%
V_{\varphi }^{\ast }\pi _{\varphi }\left( \cdot \right) V_{\varphi }$, where 
$\pi _{\varphi }$ is a representation of $A$ on a Hilbert space $\mathcal{H}%
_{\varphi }$ and $V_{\varphi }$ is a bounded linear operator (for non-unital
case, see, for example, \cite[Theorem 5.6]{L}). Given two completely
positive maps $\varphi $ and $\psi $ from a $C^{\ast }$-algebra $A$ to $L(%
\mathcal{H})$, $\psi \leq \varphi $ if $\varphi -$ $\psi $ is a completely
positive map from $A$ to $L(\mathcal{H})$. Arveson \cite[Theorem 1.4.2]{2}
showed that, in the unital case, $\psi \leq \varphi $ if and only if there
is a unique positive contraction $\Delta _{\varphi }\left( \psi \right) $ in
the commutant of $\pi _{\varphi }\left( A\right) $, such that $\psi \left(
\cdot \right) =V_{\varphi }^{\ast }\Delta _{\varphi }\left( \psi \right) \pi
_{\varphi }\left( \cdot \right) V_{\varphi }$ (for non-unital case, see, for
example, \cite[Theorem 3.5]{5}). This result can be regarded as a
Radon-Nicodym type theorem for operator valued completely positive maps on $%
C^{\ast }$-algebras and the positive linear operator $\Delta _{\varphi
}\left( \psi \right) $ is called the Radon-Nikodym derivative of $\psi $
with respect to $\varphi $. Asadi \cite{3} and Bhat, Ramesh and Sumesh \cite%
{4} provided a construction which looks like Stinespring's construction for
a class of maps on Hilbert $C^{\ast }$-modules over unital $C^{\ast }$%
-algebras, called operator valued completely positive maps on Hilbert $%
C^{\ast }$-modules. A covariant version of this construction was obtained in 
\cite{J}. In this paper, we will prove a Radon-Nicodym type theorem for
operator valued completely positive maps on Hilbert $C^{\ast }$-modules.

Hilbert $C^{\ast }$-modules are generalizations of Hilbert spaces and $%
C^{\ast }$-algebras. A Hilbert $C^{\ast }$-module $X$ over a $C^{\ast }$%
-algebra $A$ (or a Hilbert $A$-module) is a linear space that is also a
right $A$-module, equipped with an $A$-valued inner product $\left\langle
\cdot ,\cdot \right\rangle $ that is $\mathbb{C}$- and $A$-linear in the
second variable and conjugate linear in the first variable such that $X$ is
complete with the norm $\left\Vert x\right\Vert =\left\Vert \left\langle
x,x\right\rangle \right\Vert ^{\frac{1}{2}}$. If the closed bilateral $\ast $%
-sided ideal $\left\langle X,X\right\rangle $ of $A$ generated by $%
\{\left\langle x,y\right\rangle ;x,y\in X\}$ coincides with $A$, we say that 
$X$ is full.

Given two Hilbert spaces $\mathcal{H}$ and $\mathcal{K}$, the Banach space $%
L(\mathcal{H},\mathcal{K})$ of all bounded linear operators from $\mathcal{H}
$ to $\mathcal{K}$ has a canonical structure of Hilbert $C^{\ast }$-modules
over $L(\mathcal{H})$ with the right module action given by $T\cdot S=TS$
for $T\in $ $L(\mathcal{H},\mathcal{K})$ and $S\in L(\mathcal{H})$ and the
inner product given by $\left\langle T_{1},T_{2}\right\rangle =T_{1}^{\ast
}T_{2}$ for all $T_{1},T_{2}\in L(\mathcal{H},\mathcal{K}).$

\textit{A} \textit{representation }of the Hilbert $A$-module $X$ on the
Hilbert spaces $\mathcal{H}$ and $\mathcal{K}$ is a map $\pi _{X}$ $%
:X\rightarrow L(\mathcal{H},\mathcal{K})$ with the property that there is a $%
\ast $-representation $\pi _{A}$ of $A$ on $\mathcal{H}$ such that 
\begin{equation*}
\left\langle \pi _{X}\left( x\right) ,\pi _{X}\left( y\right) \right\rangle
=\pi _{A}\left( \left\langle x,y\right\rangle \right)
\end{equation*}%
for all $x$ and $y$ in $X$. If $X$ is full, then the underlying $\ast $%
-representation $\pi _{A}$ of $\pi _{X}$ is unique. A representation $\pi
_{X}:X\rightarrow L(\mathcal{H},\mathcal{K})$ of $X$ is \textit{%
nondegenerate,} if $\left[ \pi _{X}(X)\mathcal{H}\right] =\mathcal{K}$ and $%
\left[ \pi _{X}(X)^{\ast }\mathcal{K}\right] =\mathcal{H}$ (here, $[Y]$
denotes the closed subspace of a Hilbert space $Z$ generated by the subset $%
Y\subseteq Z$). Two representations $\pi _{X}:X\rightarrow L(\mathcal{H},%
\mathcal{K})$ and $\pi _{X}^{^{\prime }}:X\rightarrow L(\mathcal{H}%
^{^{\prime }},\mathcal{K}^{^{\prime }})$ are \textit{unitarily equivalent }%
if there are two unitary operators $U_{1}\in L(\mathcal{H},\mathcal{H}%
^{^{\prime }})$ and $U_{2}\in L(\mathcal{K},\mathcal{K}^{^{\prime }})$ such
that $U_{2}\pi _{X}(x)=\pi _{X}^{^{\prime }}(x)U_{1}$ for all $x$ in $X$
(see, for example, \cite{A}).

\textit{An operator valued completely positive map }on\textit{\ }$X$ is a
map $\Phi :X\rightarrow L(\mathcal{H},\mathcal{K})$ with the property that
there is a completely positive map $\varphi :A\rightarrow L(\mathcal{H})$,
such that%
\begin{equation*}
\left\langle \Phi \left( x\right) ,\Phi \left( y\right) \right\rangle
=\varphi \left( \left\langle x,y\right\rangle \right) 
\end{equation*}%
for all $x$ and $y$ in $X$. If $X$ is full, then the completely positive
linear map $\varphi $ associated to $\Phi $ is unique. Throught the paper,
when we say that $\Phi $ is an operator valued completely positive map on $X$%
, we will suppose that its associated complety positive map on $A$ is
denoted by the same small letter $\varphi $. If $\Phi :X\rightarrow L(%
\mathcal{H},\mathcal{K})$ is a completely positive map on $X$, then $\Phi $
is linear and continuous. An operator valued completely positive map\textit{%
\ } $\Phi :X\rightarrow L(\mathcal{H},\mathcal{K})$ is nondegenerate if $%
\left[ \Phi \left( X\right) \mathcal{H}\right] =\mathcal{K}$ and $\left[
\Phi \left( X\right) ^{\ast }\mathcal{K}\right] =\mathcal{H}$.

In \cite[Theorem 2.2 (1)]{J}, we showed that an operator valued completely
positive map $\Phi $ on a full Hilbert $C^{\ast }$-module $X$ is of the form 
$\Phi (\cdot )=W_{\Phi }^{\ast }\pi _{\Phi }(\cdot )V_{\Phi }$, where $\pi
_{\Phi }$ is a representation of $X$ on the Hilbert spaces $\mathcal{H}%
_{\Phi }$ and $\mathcal{K}_{\Phi }$, $W_{\Phi }$ is a coisometry from $%
\mathcal{K}_{\Phi }$ to $\mathcal{K}$ and $V_{\Phi }$ is a bounded linear
operator from $\mathcal{H}$ to $\mathcal{H}_{\Phi }$. Quintuple $\left( \pi
_{\Phi },\mathcal{H}_{\Phi },\mathcal{K}_{\Phi },V_{\Phi },W_{\Phi }\right) $
is called the Stinespring construction associated to $\Phi $. Moreover, $%
\left( \pi _{\varphi },\mathcal{H}_{\Phi },V_{\Phi }\right) $, where $\pi
_{\varphi }$ is the $\ast $-representation of $A$ associated to $\pi _{\Phi
},$ is the Stinespring construction associated to $\varphi $, and under some
conditions the Stinespring construction associated to $\Phi $ is unique up
to unitary equivalence (see \cite[Theorem 2.2 (2)]{J}) in the sense that if $%
\left( \pi ^{\prime },\mathcal{H}^{\prime },\mathcal{K}^{\prime },V^{\prime
},W^{\prime }\right) $ is another quintuple such that $\Phi (\cdot )=\left(
W^{\prime }\right) ^{\ast }\pi ^{\prime }(\cdot )V^{\prime },\left[ \pi
^{\prime }(X)V^{\prime }\mathcal{H}\right] =\mathcal{H}^{\prime }$ and $%
\left[ \pi ^{\prime }(X)^{\ast }W^{\prime }\mathcal{K}\right] =\mathcal{K}%
^{\prime }$, then there are two unitary operators $U_{1}\in L(\mathcal{H}%
_{\Phi },\mathcal{H}^{\prime })$ and $U_{2}\in L(\mathcal{K}_{\Phi },%
\mathcal{K}^{\prime })$ such that, $U_{1}V_{\Phi }=V^{\prime }$, $%
U_{2}W_{\Phi }=W^{\prime }$ and $U_{2}\pi _{\Phi }\left( x\right) =\pi
^{\prime }(x)U_{1}$ for all $x\in X$.

In this paper, we introduce an equivalence relation on the collection of all
operator valued completely positive maps on a full Hilbert $C^{\ast }$%
-module $X$, and we show that the Stinespring constructions associated to
equivalent completely positive maps are unitarily equivalent. Also, we
introduce a preorder relation in the collection of all operator valued
completely positive maps on a full Hilber $C^{\ast }$-module $X$, and prove
a Radon-Nicodym type theorem for operator valued completely positive maps on
Hilbert $C^{\ast }$-modules.

\section{Completely positive maps}

Let $X$ be a full Hilbert $C^{\ast }$-module over a $C^{\ast }$-algebra $A$,
let $\mathcal{H}$ and $\mathcal{K}$ be two Hilbert spaces and $\mathcal{C}%
(X,L(\mathcal{H},\mathcal{K}))=\{\Phi :X\rightarrow L(\mathcal{H},\mathcal{K}%
)$; $\Phi $ is completely positive$\}$.

\begin{definition}
Let $\Phi ,\Psi \in \mathcal{C}(X,L(\mathcal{H},\mathcal{K}))$. We say that $%
\Phi $ is equivalent to $\Psi $, denoted by $\Phi \backsim \Psi $, if $\Phi
\left( x\right) ^{\ast }\Phi \left( x\right) =\Psi \left( x\right) ^{\ast
}\Psi \left( x\right) $ for all $x\in X$.
\end{definition}

\begin{remark}
The relation defined above is an equivalence relation on $\mathcal{C}(X,L(%
\mathcal{H},\mathcal{K}))$.
\end{remark}

\begin{proposition}
Let $\Phi ,\Psi \in \mathcal{C}(X,L(\mathcal{H},\mathcal{K}))$. Then $\Phi
\backsim \Psi $ if and only if there is a partial isometry $V\in L(\mathcal{K%
})$ with $VV^{\ast }=p_{[\Phi \left( X\right) \mathcal{H}]}$ and $V^{\ast
}V=p_{[\Psi (X)\mathcal{H}]}$ such that $\Phi \left( x\right) =V\Psi \left(
x\right) $ for all $x\in X$.
\end{proposition}

\begin{proof}
First, we suppose that $\Phi \backsim \Psi $. Let $\left( \pi _{\Phi },%
\mathcal{H}_{\Phi },\mathcal{K}_{\Phi },V_{\Phi },W_{\Phi }\right) $ be the
Stinesping construction associated to $\Phi $ and $\left( \pi _{\Psi },%
\mathcal{H}_{\Psi },\mathcal{K}_{\Psi },V_{\Psi },W_{\Psi }\right) $ the
Stinespring construction associated to $\Psi $. Since $\Phi \left( x\right)
^{\ast }\Phi \left( x\right) =\Psi \left( x\right) ^{\ast }\Psi \left(
x\right) $ for all $x\in X$, $\varphi =\psi $ and then, by the proof of
Theorem 2.2 in \cite{J} and Lemma 1.4.1 in \cite{2}, there is a unitary
operator $U_{1}\in L(\mathcal{H}_{\Psi },\mathcal{H}_{\Phi })$ such that $%
V_{\Phi }=U_{1}V_{\Psi }$.

From 
\begin{eqnarray*}
\left\langle \pi _{\Phi }\left( x\right) V_{\Phi }h,\pi _{\Phi }\left(
x\right) V_{\Phi }h\right\rangle &=&\left\langle V_{\Phi }^{\ast }\pi _{\Phi
}\left( x\right) ^{\ast }\pi _{\Phi }\left( x\right) V_{\Phi
}h,h\right\rangle \\
&=&\left\langle V_{\Phi }^{\ast }\pi _{\varphi }\left( \left\langle
x,x\right\rangle \right) V_{\Phi }h,h\right\rangle \\
&=&\left\langle \varphi \left( \left\langle x,x\right\rangle \right)
h,h\right\rangle =\left\langle \psi \left( \left\langle x,x\right\rangle
\right) h,h\right\rangle \\
&=&\left\langle \pi _{\Psi }\left( x\right) V_{\Psi }h,\pi _{\Psi }\left(
x\right) V_{\Psi }h\right\rangle
\end{eqnarray*}%
for all $x\in X$ and for all $h\in \mathcal{H}$, and taking into account
that $\mathcal{K}_{\Psi }=[\Psi \left( x\right) \mathcal{H}]=[\pi _{\Psi
}\left( X\right) V_{\Psi }\mathcal{H}]$ and $\mathcal{K}_{\Phi }=[\Phi
\left( x\right) \mathcal{H}]=[\pi _{\Phi }\left( X\right) V_{\Phi }\mathcal{H%
}]$ (see \cite[Theorem 2.2]{J}), we deduce that there is a unitary operator $%
U_{2}:\mathcal{K}_{\Psi }\rightarrow \mathcal{K}_{\Phi }$, such that 
\begin{equation*}
U_{2}\left( \pi _{\Psi }\left( x\right) V_{\Psi }h\right) =\pi _{\Phi
}\left( x\right) V_{\Phi }h
\end{equation*}%
for all $h\in \mathcal{H}$. Moreover, $U_{2}\pi _{\Psi }(x)=\pi _{\Phi
}\left( x\right) U_{1}$ for all $x\in X$, since 
\begin{eqnarray*}
U_{2}\pi _{\Psi }(x)\left( \pi _{\psi }\left( a\right) V_{\Psi }h\right)
&=&U_{2}\left( \pi _{\Psi }\left( xa\right) V_{\Psi }h\right) =\pi _{\Phi
}\left( xa\right) V_{\Phi }h \\
&=&\pi _{\Phi }\left( x\right) \left( \pi _{\varphi }\left( a\right) V_{\Phi
}h\right) =\pi _{\Phi }\left( x\right) U_{1}\left( \pi _{\psi }\left(
a\right) V_{\Psi }h\right)
\end{eqnarray*}%
for all $a\in A$ and for all $h\in \mathcal{H}$, and since $[\pi _{\psi
}\left( A\right) V_{\Psi }\mathcal{H}]=\mathcal{H}_{\Psi }$.

Let $V=W_{\Phi }^{\ast }U_{2}W_{\Psi }$. From 
\begin{equation*}
VV^{\ast }=W_{\Phi }^{\ast }U_{2}W_{\Psi }W_{\Psi }^{\ast }U_{2}^{\ast
}W_{\Phi }=p_{\mathcal{K}_{\Phi }}=p_{[\Phi \left( X\right) \mathcal{H}]}
\end{equation*}%
and 
\begin{equation*}
V^{\ast }V=W_{\Psi }^{\ast }U_{2}^{\ast }W_{\Phi }W_{\Phi }^{\ast
}U_{2}W_{\Psi }=p_{\mathcal{K}_{\Psi }}=p_{[\Psi (X)\mathcal{H}]}\text{,}
\end{equation*}%
we deduce that $V$ is a partial isometry. Moreover,%
\begin{eqnarray*}
\Phi \left( x\right) &=&W_{\Phi }^{\ast }\pi _{\Phi }(x)V_{\Phi }=W_{\Phi
}^{\ast }\pi _{\Phi }(x)U_{1}V_{\Psi }=W_{\Phi }^{\ast }U_{2}\pi _{\Psi
}(x)V_{\Psi } \\
&=&W_{\Phi }^{\ast }U_{2}W_{\Psi }W_{\Psi }^{\ast }\pi _{\Psi }(x)V_{\Psi
}=W_{\Phi }^{\ast }U_{2}W_{\Psi }\Psi \left( x\right) =V\Psi \left( x\right)
\end{eqnarray*}%
for all $x\in X$.

Conversely, suppose that there is a partial isometry $V\in L(\mathcal{K})$
with $VV^{\ast }=p_{[\Phi \left( X\right) \mathcal{H}]}$ and $V^{\ast
}V=p_{[\Psi (X)\mathcal{H}]}$ such that $\Phi \left( x\right) =V\Psi \left(
x\right) $ for all $x\in X$. Then $\Phi \left( x\right) ^{\ast }\Phi \left(
x\right) =\Psi \left( x\right) ^{\ast }V^{\ast }V\Psi \left( x\right) =\Psi
\left( x\right) ^{\ast }\Psi \left( x\right) $ for all $x\in X$, and so $%
\Phi \backsim \Psi $.
\end{proof}

\begin{remark}
Let $\Phi ,\Psi \in \mathcal{C}(X,L(\mathcal{H},\mathcal{K}))$ such that $%
\Phi \backsim \Psi $.$\ $If $\Phi ,\Psi $ are nondegenerate, then there is a
unitary operator $V\in L(\mathcal{K})$ such that $\Phi \left( x\right)
=V\Psi \left( x\right) \ $for all $x\in X$. Indeed, if $\Phi $ and $\Psi $
are nondegenerate, then $[\Phi \left( X\right) \mathcal{H}]=[\Psi \left(
X\right) \mathcal{H}]=\mathcal{K}$, and so $V$ is a unitary operator.
\end{remark}

\begin{corollary}
Let $\Phi ,\Psi \in \mathcal{C}(X,L(\mathcal{H},\mathcal{K}))$. Then $\Phi
\backsim \Psi $ if and only if their Stinespring constructions are unitarily
equivalent.
\end{corollary}

\begin{proof}
Let $\left( \pi _{\Phi },\mathcal{H}_{\Phi },\mathcal{K}_{\Phi },V_{\Phi
},W_{\Phi }\right) $ and $\left( \pi _{\Psi },\mathcal{H}_{\Psi },\mathcal{K}%
_{\Psi },V_{\Psi },W_{\Psi }\right) $ be the Stinesping constructions
associated to $\Phi $ and $\Psi $. If $U_{2}$ and $U_{1}$ are the unitary
operators defined in the proof of Proposition 2.3, then: $U_{2}\pi _{\Psi
}(x)=\pi _{\Phi }\left( x\right) U_{1}$ for all $x\in X;V_{\Phi
}=U_{1}V_{\Psi }$, and it is not difficult to check that $W_{\Phi
}=U_{2}W_{\Psi }V^{\ast }$. Therefore, $\left( \pi _{\Phi },\mathcal{H}%
_{\Phi },\mathcal{K}_{\Phi },V_{\Phi },W_{\Phi }\right) $ and $\left( \pi
_{\Psi },\mathcal{H}_{\Psi },\mathcal{K}_{\Psi },V_{\Psi },W_{\Psi }\right) $
are unitarily equivalent.

Clearly, if the Stinespring constructions associated to $\Phi $ and $\Psi $
are unitarily equivalent, then $\Phi \left( x\right) ^{\ast }\Phi \left(
x\right) =\Psi \left( x\right) ^{\ast }\Psi \left( x\right) $ for all $x\in
X $, and so $\Phi \backsim \Psi $.
\end{proof}

\begin{example}
It is not difficult to verify that the maps $\Phi ,\Psi :M_{2}(\mathbb{C}%
)\rightarrow L(\mathbb{C}^{2},\mathbb{C}^{4})$ defined by 
\begin{equation*}
\Phi \left( \left[ 
\begin{array}{cc}
a_{11} & a_{12} \\ 
a_{21} & a_{22}%
\end{array}%
\right] \right) =\left[ 
\begin{array}{cc}
\frac{\sqrt{3}}{2}a_{11} & \frac{\sqrt{3}}{2}a_{12} \\ 
\frac{\sqrt{3}}{2}a_{21} & \frac{\sqrt{3}}{2}a_{22} \\ 
\frac{1}{2}a_{11} & \frac{-1}{2}a_{12} \\ 
\frac{1}{2}a_{21} & \frac{-1}{2}a_{22}%
\end{array}%
\right] ,
\end{equation*}%
respectively 
\begin{equation*}
\Psi \left( \left[ 
\begin{array}{cc}
a_{11} & a_{12} \\ 
a_{21} & a_{22}%
\end{array}%
\right] \right) =\left[ 
\begin{array}{cc}
\frac{\sqrt{3}}{2}a_{11} & \frac{\sqrt{3}}{2}a_{12} \\ 
\frac{\sqrt{3}}{2}a_{21} & \frac{\sqrt{3}}{2}a_{22} \\ 
\frac{-1}{2}a_{11} & \frac{1}{2}a_{12} \\ 
\frac{1}{2}a_{21} & \frac{-1}{2}a_{22}%
\end{array}%
\right]
\end{equation*}%
are completely positive with the same underling map $\varphi :M_{2}(\mathbb{C%
})\rightarrow M_{2}(\mathbb{C})\ $given by $\varphi \left( \left[ 
\begin{array}{cc}
a_{11} & a_{12} \\ 
a_{21} & a_{22}%
\end{array}%
\right] \right) =\left[ 
\begin{array}{cc}
a_{11} & \frac{a_{12}}{2} \\ 
\frac{a_{21}}{2} & a_{22}%
\end{array}%
\right] $. So $\Phi \backsim \Psi $. Moreover, $\Phi $ and $\Psi $ are
nondegenerate and 
\begin{equation*}
\Phi \left( \left[ 
\begin{array}{cc}
a_{11} & a_{12} \\ 
a_{21} & a_{22}%
\end{array}%
\right] \right) =\left[ 
\begin{array}{cccc}
1 & 0 & 0 & 0 \\ 
0 & 1 & 0 & 0 \\ 
0 & 0 & -1 & 0 \\ 
0 & 0 & 0 & 1%
\end{array}%
\right] \Psi \left( \left[ 
\begin{array}{cc}
a_{11} & a_{12} \\ 
a_{21} & a_{22}%
\end{array}%
\right] \right) .
\end{equation*}
\end{example}

\begin{example}
A simple calculus shows that the maps $\Phi ,\Psi :M_{2}(\mathbb{C}%
)\rightarrow L(\mathbb{C}^{2},\mathbb{C}^{5})$ defined by 
\begin{equation*}
\Phi \left( \left[ 
\begin{array}{cc}
a_{11} & a_{12} \\ 
a_{21} & a_{22}%
\end{array}%
\right] \right) =\left[ 
\begin{array}{cc}
\sqrt{2}a_{11} & 0 \\ 
0 & \sqrt{3}a_{22} \\ 
0 & 0 \\ 
\sqrt{2}a_{21} & 0 \\ 
0 & \sqrt{3}a_{12}%
\end{array}%
\right] ,
\end{equation*}%
respectively 
\begin{equation*}
\Psi \left( \left[ 
\begin{array}{cc}
a_{11} & a_{12} \\ 
a_{21} & a_{22}%
\end{array}%
\right] \right) =\left[ 
\begin{array}{cc}
a_{11} & -a_{22} \\ 
a_{11} & a_{22} \\ 
0 & \sqrt{3}a_{12} \\ 
\sqrt{2}a_{21} & 0 \\ 
0 & a_{22}%
\end{array}%
\right]
\end{equation*}%
are completely positive with the same underling map $\varphi :M_{2}(\mathbb{C%
})\rightarrow M_{2}(\mathbb{C})\ $given by $\varphi \left( \left[ 
\begin{array}{cc}
a_{11} & a_{12} \\ 
a_{21} & a_{22}%
\end{array}%
\right] \right) =\left[ 
\begin{array}{cc}
2a_{11} & 0 \\ 
0 & 3a_{22}%
\end{array}%
\right] $, and $\Phi $ is degenerate. So $\Phi \backsim \Psi $, and 
\begin{equation*}
\Phi \left( \left[ 
\begin{array}{cc}
a_{11} & a_{12} \\ 
a_{21} & a_{22}%
\end{array}%
\right] \right) =\left[ 
\begin{array}{ccccc}
\frac{\sqrt{2}}{2} & \frac{\sqrt{2}}{2} & 0 & 0 & 0 \\ 
-\frac{\sqrt{3}}{3} & \frac{\sqrt{3}}{3} & 0 & 0 & \frac{\sqrt{3}}{3} \\ 
0 & 0 & 0 & 0 & 0 \\ 
0 & 0 & 0 & 1 & 0 \\ 
0 & 0 & 1 & 0 & 0%
\end{array}%
\right] \Psi \left( \left[ 
\begin{array}{cc}
a_{11} & a_{12} \\ 
a_{21} & a_{22}%
\end{array}%
\right] \right) .
\end{equation*}
\end{example}

\begin{definition}
Let $\Phi ,\Psi \in \mathcal{C}(X,L(\mathcal{H},\mathcal{K}))$. We say that $%
\Phi $ is dominate by $\Psi $, denoted by $\Phi \curlyeqprec \Psi $, if $\
\varphi \leq \psi \ $(in the sense that the $\psi -\varphi $ is a completely
positive  linear map) where  $\varphi $ and $\psi $ are the completely
positive linear maps associated to $\Phi $ respectively $\Psi .$
\end{definition}

\begin{remark}
Let $\Phi ,\Psi ,\Phi _{1},\Phi _{2},\Phi _{3}\in \mathcal{C}(X,L(\mathcal{H}%
,\mathcal{K}))$. Then we have:

\begin{enumerate}
\item $\Phi \curlyeqprec \Phi $ for all $\Phi \in $ $\mathcal{C}(X,L(%
\mathcal{H},\mathcal{K}));$

\item If $\Phi _{1}\curlyeqprec \Phi _{2}$ and $\Phi _{2}\curlyeqprec \Phi
_{3}$, then $\Phi _{1}\curlyeqprec \Phi _{3};$

\item $\Phi \curlyeqprec \Psi $ and $\Psi \curlyeqprec \Phi $ if and only if 
$\Phi \backsim \Psi $.
\end{enumerate}
\end{remark}

For a representation $\pi _{X}$ of $X$ on the Hilbert spaces $\mathcal{H}$
and $\mathcal{K},$ $\pi _{X}\left( X\right) ^{\prime }=\{T\oplus S\in L(%
\mathcal{H}\oplus \mathcal{K});\pi _{X}(x)T=S\pi _{X}(x)$ and $\pi
_{X}(x)^{\ast }S=T\pi _{X}(x)^{\ast }$ for all $x\in X\}$ is a $C^{\ast }$%
-algebra called the commutant of $\pi _{X}$ (see \cite[Lemma 4.3]{A}). If $%
\left( \pi _{X},\mathcal{H},\mathcal{K}\right) $ is nondegenerate, then $%
T\in \pi _{A}\left( A\right) ^{\prime }$ (see \cite[Lemma 4.4]{A}), and
moreover, if $T\oplus S\in \pi _{X}\left( X\right) ^{\prime }$, then $S$ is
unique determined by $T$ (see \cite[Note 4.6]{A}).

\begin{lemma}
Let $\Phi \in \mathcal{C}(X,L(\mathcal{H},\mathcal{K}))$ and let $\left( \pi
_{\Phi },\mathcal{H}_{\Phi },\mathcal{K}_{\Phi },V_{\Phi },W_{\Phi }\right) $
be the Stinesping construction associated to $\Phi $. If $T\oplus S\in \pi
_{\Phi }\left( X\right) ^{\prime }$ is a positive element, then the map $%
\Phi _{T\oplus S}:X\rightarrow L(\mathcal{H},\mathcal{K})$, defined by $\Phi
_{T\oplus S}\left( x\right) =W_{\Phi }^{\ast }\sqrt{S}\pi _{\Phi }(x)\sqrt{T}%
V_{\Phi }$, is completely positive.
\end{lemma}

\begin{proof}
Indeed, we have 
\begin{eqnarray*}
\Phi _{T\oplus S}\left( x\right) ^{\ast }\Phi _{T\oplus S}\left( y\right)
&=&V_{\Phi }^{\ast }\sqrt{T}\pi _{\Phi }(x)^{\ast }\sqrt{S}W_{\Phi }W_{\Phi
}^{\ast }\sqrt{S}\pi _{\Phi }(y)\sqrt{T}V_{\Phi } \\
&=&V_{\Phi }^{\ast }\sqrt{T}\pi _{\Phi }(x)^{\ast }S\pi _{\Phi }(y)\sqrt{T}%
V_{\Phi } \\
&=&V_{\Phi }^{\ast }\sqrt{T}\pi _{\Phi }(x)^{\ast }S\sqrt{S}\pi _{\Phi
}(y)V_{\Phi }=V_{\Phi }^{\ast }T^{2}\pi _{\Phi }(x)^{\ast }\pi _{\Phi
}(y)V_{\Phi } \\
&=&V_{\Phi }^{\ast }T^{2}\pi _{\varphi }(\left\langle x,y\right\rangle
)V_{\Phi } \\
&&\text{(cf. \cite[Theorem 1.4.2]{2})} \\
&=&\varphi _{T^{2}}\left( \left\langle x,y\right\rangle \right)
\end{eqnarray*}%
for all $x,y\in X$, and so $\Phi _{T\oplus S}$ is an operator valued
completely positive map.
\end{proof}

\begin{remark}
The completely positive map associated to $\Phi _{T\oplus S}$ is $\varphi
_{T^{2}}$.
\end{remark}

\begin{theorem}
Let $\Psi ,\Phi \in \mathcal{C}(X,L(\mathcal{H},\mathcal{K}))$. If $\ \Psi
\curlyeqprec \Phi $, then there is a unique positive linear operator $\Delta
_{\Phi }\left( \Psi \right) $ in $\pi _{\Phi }\left( X\right) ^{\prime }$
such that $\Psi \backsim \Phi _{\sqrt{\Delta _{\Phi }\left( \Psi \right) }}$.
\end{theorem}

\begin{proof}
Let $\left( \pi _{\Phi },\mathcal{H}_{\Phi },\mathcal{K}_{\Phi },V_{\Phi
},W_{\Phi }\right) $ and $\left( \pi _{\Psi },\mathcal{H}_{\Psi },\mathcal{K}%
_{\Psi },V_{\Psi },W_{\Psi }\right) $ be the Stinspring constructions
associated to $\Phi $ and $\Psi $. If $\ \Psi \curlyeqprec \Phi $, then $%
\psi \leq \varphi $, and by the proof of \cite[Lemma 1.4.1]{2}, there is a
bounded linear operator $J_{\Phi }\left( \Psi \right) :\mathcal{H}_{\Phi
}\rightarrow \mathcal{H}_{\Psi }$, such that 
\begin{equation*}
J_{\Phi }\left( \Psi \right) \left( \pi _{\varphi }\left( a\right) V_{\Phi
}h\right) =\pi _{\psi }\left( a\right) V_{\Psi }h
\end{equation*}%
for all $a\in A$ and for all $h\in \mathcal{H}$. Moreover, $\left\Vert
J_{\Phi }\left( \Psi \right) \right\Vert \leq 1$ and $\psi \left( a\right)
=\varphi _{J_{\Phi }\left( \Psi \right) ^{\ast }J_{\Phi }\left( \Psi \right)
}\left( a\right) $ for all $a\in A$. Since%
\begin{eqnarray*}
\left\langle \tsum_{i=1}^{n}\pi _{\Psi }\left( x_{i}\right) V_{\Psi
}h_{i},\tsum_{i=1}^{n}\pi _{\Psi }\left( x_{i}\right) V_{\Psi
}h_{i}\right\rangle  &=&\tsum_{i,j=1}^{n}\left\langle h_{i},V_{\Psi }^{\ast
}\pi _{\Psi }\left( x_{i}\right) ^{\ast }\pi _{\Psi }\left( x_{j}\right)
V_{\Psi }h_{j}\right\rangle  \\
&=&\tsum_{i,j=1}^{n}\left\langle h_{i},V_{\Psi }^{\ast }\pi _{\Psi }\left(
x_{i}\right) ^{\ast }W_{\Psi }W_{\Psi }^{\ast }\pi _{\Psi }\left(
x_{j}\right) V_{\Psi }h_{j}\right\rangle  \\
&=&\tsum_{i,j=1}^{n}\left\langle h_{i},\Psi \left( x_{i}\right) ^{\ast }\Psi
\left( x_{j}\right) h_{j}\right\rangle =\tsum_{i,j=1}^{n}\left\langle
h_{i},\psi \left( \left\langle x_{i},x_{j}\right\rangle \right)
h_{j}\right\rangle  \\
&\leq &\tsum_{i,j=1}^{n}\left\langle h_{i},\varphi \left( \left\langle
x_{i},x_{j}\right\rangle \right) h_{j}\right\rangle
=\tsum_{i,j=1}^{n}\left\langle h_{i},\Phi \left( x_{i}\right) ^{\ast }\Phi
\left( x_{j}\right) h_{j}\right\rangle  \\
&=&\left\langle \tsum_{i=1}^{n}\pi _{\Phi }\left( x_{i}\right) V_{\Phi
}h_{i},\tsum_{i=1}^{n}\pi _{\Phi }\left( x_{i}\right) V_{\Phi
}h_{i}\right\rangle 
\end{eqnarray*}%
for all $x_{1},...,x_{n}\in X$ and for all $h_{1},...,h_{n}\in H$, and since 
$\left[ \pi _{\Psi }\left( X\right) V_{\Psi }H\right] =K_{\Psi }$, there is
a bounded linear operator $I_{\Phi }\left( \Psi \right) :K_{\Phi
}\rightarrow K_{\Psi }$, such that 
\begin{equation*}
I_{\Phi }\left( \Psi \right) \left( \pi _{\Phi }\left( x\right) V_{\Phi
}h\right) =\pi _{\Psi }\left( x\right) V_{\Psi }h
\end{equation*}%
for all $x\in X$ and for all $h\in H$, and $\left\Vert I_{\Phi }\left( \Psi
\right) \right\Vert \leq 1$.

Let $x\in X$. From%
\begin{eqnarray*}
\left( I_{\Phi }\left( \Psi \right) \pi _{\Phi }\left( x\right) \right)
\left( \pi _{\varphi }\left( a\right) V_{\Phi }h\right) &=&I_{\Phi }\left(
\Psi \right) \left( \pi _{\Phi }\left( xa\right) V_{\Phi }h\right) =\pi
_{\Psi }\left( xa\right) V_{\Psi }h \\
&=&\pi _{\Psi }\left( x\right) \left( \pi _{\psi }\left( a\right) V_{\Psi
}h\right) \\
&=&\pi _{\Psi }\left( x\right) J_{\Phi }\left( \Psi \right) \left( \pi
_{\varphi }\left( a\right) V_{\Phi }h\right)
\end{eqnarray*}%
for all $a\in A$ and for all $h\in \mathcal{H}$, and taking into account
that $\left[ \left( \pi _{\varphi }\left( A\right) V_{\Phi }\mathcal{H}%
\right) \right] =\mathcal{H}_{\Phi }$, we deduce that $I_{\Phi }\left( \Psi
\right) \pi _{\Phi }\left( x\right) =\pi _{\Psi }\left( x\right) J_{\Phi
}\left( \Psi \right) $, and from 
\begin{eqnarray*}
\left( \pi _{\Psi }\left( x\right) ^{\ast }I_{\Phi }\left( \Psi \right)
\right) \left( \pi _{\Phi }\left( y\right) V_{\Phi }h\right) &=&\pi _{\Psi
}\left( x\right) ^{\ast }\pi _{\Psi }\left( y\right) V_{\Psi }h=\pi _{\psi
}\left( \left\langle x,y\right\rangle \right) V_{\Psi }h \\
&=&J_{\Phi }\left( \Psi \right) \left( \pi _{\varphi }\left( \left\langle
x,y\right\rangle \right) V_{\Phi }h\right) \\
&=&J_{\Phi }\left( \Psi \right) \pi _{\Phi }\left( x\right) ^{\ast }\left(
\pi _{\Phi }\left( y\right) V_{\Phi }h\right)
\end{eqnarray*}%
for all $y\in X$ and for all $h\in \mathcal{H}$, and taking into account
that $\left[ \left( \pi _{\Phi }\left( X\right) V_{\Phi }\mathcal{H}\right) %
\right] =\mathcal{K}_{\Phi }$, we deduce that $\pi _{\Psi }\left( x\right)
^{\ast }I_{\Phi }\left( \Psi \right) =J_{\Phi }\left( \Psi \right) \pi
_{\Phi }\left( x\right) ^{\ast }$.

Let $\Delta _{\Phi }\left( \Psi \right) =\Delta _{1\Phi }\left( \Psi \right)
\oplus \Delta _{2\Phi }\left( \Psi \right) $, where $\Delta _{1\Phi }\left(
\Psi \right) =J_{\Phi }\left( \Psi \right) ^{\ast }J_{\Phi }\left( \Psi
\right) $ and $\Delta _{2\Phi }\left( \Psi \right) =$ $I_{\Phi }\left( \Psi
\right) ^{\ast }I_{\Phi }\left( \Psi \right) $. Then we have:%
\begin{eqnarray*}
\Delta _{2\Phi }\left( \Psi \right) \pi _{\Phi }\left( x\right) &=&I_{\Phi
}\left( \Psi \right) ^{\ast }I_{\Phi }\left( \Psi \right) \pi _{\Phi }\left(
x\right) =I_{\Phi }\left( \Psi \right) ^{\ast }\pi _{\Psi }\left( x\right)
J_{\Phi }\left( \Psi \right) \\
&=&\pi _{\Phi }\left( x\right) J_{\Phi }\left( \Psi \right) ^{\ast }J_{\Phi
}\left( \Psi \right) =\pi _{\Phi }\left( x\right) \Delta _{1\Phi }\left(
\Psi \right)
\end{eqnarray*}%
and 
\begin{eqnarray*}
\pi _{\Phi }\left( x\right) ^{\ast }\Delta _{2\Phi }\left( \Psi \right)
&=&\pi _{\Phi }\left( x\right) ^{\ast }I_{\Phi }\left( \Psi \right) ^{\ast
}I_{\Phi }\left( \Psi \right) =J_{\Phi }\left( \Psi \right) ^{\ast }\pi
_{\Psi }\left( x\right) ^{\ast }I_{\Phi }\left( \Psi \right) \\
&=&J_{\Phi }\left( \Psi \right) ^{\ast }J_{\Phi }\left( \Psi \right) \pi
_{\Phi }\left( x\right) ^{\ast }=\Delta _{1\Phi }\left( \Psi \right) \pi
_{\Phi }\left( x\right) ^{\ast }
\end{eqnarray*}%
for all $x\in X$. Therefore, $\Delta _{\Phi }\left( \Psi \right) \in \pi
_{\Phi }\left( X\right) ^{\prime }$ and $0\leq \Delta _{\Phi }\left( \Psi
\right) \leq I$.$\ $Moreover, 
\begin{equation*}
\Phi _{\sqrt{\Delta _{\Phi }\left( \Psi \right) }}\left( x\right) ^{\ast
}\Phi _{\sqrt{\Delta _{\Phi }\left( \Psi \right) }}\left( x\right) =\varphi
_{\Delta _{1\Phi }\left( \Psi \right) }\left( \left\langle x,x\right\rangle
\right) =\psi \left( \left\langle x,x\right\rangle \right) =\Psi \left(
x\right) ^{\ast }\Psi (x)
\end{equation*}%
for all $x\in X$, and so $\Psi \backsim \Phi _{\sqrt{\Delta _{\Phi }\left(
\Psi \right) }}$.

Suppose that there is another positive linear operator $T\oplus S\in \pi
_{\Phi }\left( X\right) ^{\prime }$ such that $\Psi \backsim \Phi _{\sqrt{%
T\oplus S}}$. Then $\Phi _{\sqrt{\Delta _{\Phi }\left( \Psi \right) }%
}\backsim \Phi _{\sqrt{T\oplus S}}$, whence we deduce that $\varphi _{\Delta
_{1\Phi }\left( \Psi \right) }=\varphi _{T}$, and by \cite[Theorem 1.4.2]{2}%
, $\Delta _{1\Phi }\left( \Psi \right) =T$, and so $\Delta _{\Phi }\left(
\Psi \right) =$ $T\oplus S$, since $\pi _{\Phi }$ is nondegenerate (see \cite%
[Note 4.6]{A}).
\end{proof}

The positive linear operator $\Delta _{\Phi }\left( \Psi \right) =\Delta
_{1\Phi }\left( \Psi \right) \oplus \Delta _{2\Phi }\left( \Psi \right) $
will be called \textit{the Radon-Nicodym derivative of \ }$\Psi $\textit{\
with respect to }$\Phi $.

\begin{remark}
\begin{enumerate}
\item If $\Delta _{\Phi }\left( \Psi \right) =\Delta _{1\Phi }\left( \Psi
\right) \oplus \Delta _{2\Phi }\left( \Psi \right) $ is the \textit{%
Radon-Nicodym derivative of \ }$\Psi $\textit{\ with respect to }$\Phi $,
then $\Delta _{1\Phi }\left( \Psi \right) $ is the Radon-Nicodym\textit{\
derivative of }$\psi $\textit{\ with respect to }$\varphi $.

\item If $\Psi _{1}\curlyeqprec \Phi ,$ $\Psi _{2}\curlyeqprec \Phi $ and $%
\Psi _{1}\backsim \Psi _{2}$, then $\Delta _{\Phi }\left( \Psi _{1}\right)
=\Delta _{\Phi }\left( \Psi _{2}\right) .$
\end{enumerate}
\end{remark}

\begin{remark}
If $\Psi \curlyeqprec \Phi $, then the Stinespring construction associated
to $\Psi $ can be recovered by the Stinespring construction associated to $%
\Phi $. Indeed, let $\left( \pi _{\Phi },\mathcal{H}_{\Phi },\mathcal{K}%
_{\Phi },V_{\Phi },W_{\Phi }\right) $ be the Stinespring construction
associated to $\Phi $ and $\Delta _{\Phi }\left( \Psi \right) =\Delta
_{1\Phi }\left( \Psi \right) \oplus \Delta _{2\Phi }\left( \Psi \right) $
the \textit{Radon-Nicodym derivative of }$\Psi $\textit{\ with respect to }$%
\Phi $. Then, the pair $\left( \ker \Delta _{1\Phi }\left( \Psi \right)
,\ker \Delta _{2\Phi }\left( \Psi \right) \right) $ is invariant under $\pi
_{\Phi }$, and so $p_{\text{ker}\Delta _{1\Phi }\left( \Psi \right) }\oplus
p_{\ker \Delta _{2\Phi }\left( \Psi \right) },p_{\mathcal{H}_{\Phi }\ominus 
\text{ker}\Delta _{1\Phi }\left( \Psi \right) }\oplus p_{\mathcal{K}_{\Phi
}\ominus \ker \Delta _{2\Phi }\left( \Psi \right) }\in \pi _{\Phi
}(X)^{\prime }$. Let $p_{1}=p_{\mathcal{H}_{\Phi }\ominus \text{ker}\Delta
_{1\Phi }\left( \Psi \right) }$ and $p_{2}=p_{K_{\Phi }\ominus \ker \Delta
_{2\Phi }\left( \Psi \right) }$. Then, we have:

\begin{enumerate}
\item $p_{1}\sqrt{\Delta _{1\Phi }\left( \Psi \right) }V_{\Phi }\in L(%
\mathcal{H},\mathcal{H}_{\Phi }\ominus $ker$\Delta _{1\Phi }\left( \Psi
\right) );$

\item $p_{2}\pi _{\Phi }\left( \cdot \right) p_{1}$ is a representation of $%
X $ on the Hilbert spaces $\mathcal{H}_{\Phi }\ominus \ker \Delta _{1\Phi
}\left( \Psi \right) $ and $\mathcal{K}_{\Phi }\ominus \ker \Delta _{2\Phi
}\left( \Psi \right) ,$ since%
\begin{equation*}
\left( p_{2}\pi _{\Phi }\left( x\right) p_{1}\right) ^{\ast }p_{2}\pi _{\Phi
}\left( y\right) p_{1}=p_{1}\pi _{\varphi }\left( \left\langle
x,y\right\rangle \right) p_{1}
\end{equation*}%
for all $x,y\in X;$

\item $p_{2}W_{\Phi }\in L(\mathcal{K},\mathcal{K}_{\Phi }\ominus \ker
\Delta _{2\Phi }\left( \Psi \right) )$ is a co-isometry, since 
\begin{equation*}
\left( p_{2}W_{\Phi }\right) \left( p_{2}W_{\Phi }\right) ^{\ast
}=p_{2}W_{\Phi }W_{\Phi }^{\ast }p_{2}=p_{\mathcal{K}_{\Phi }\ominus \ker
\Delta _{2\Phi }\left( \Psi \right) };
\end{equation*}

\item 
\begin{eqnarray*}
\left[ p_{2}\pi _{\Phi }\left( X\right) p_{1}\left( p_{1}\sqrt{\Delta
_{1\Phi }\left( \Psi \right) }V_{\Phi }\right) \mathcal{H}\right] &=&\left[
p_{2}\pi _{\Phi }\left( X\right) \sqrt{\Delta _{1\Phi }\left( \Psi \right) }%
V_{\Phi }\mathcal{H}\right] \\
&=&\left[ p_{2}\sqrt{\Delta _{2\Phi }\left( \Psi \right) }\pi _{\Phi }\left(
X\right) V_{\Phi }\mathcal{H}\right] \\
&=&\left[ p_{2}\sqrt{\Delta _{2\Phi }\left( \Psi \right) }\mathcal{K}_{\Phi }%
\right] =\mathcal{K}_{\Phi }\ominus \ker \Delta _{2\Phi }\left( \Psi \right)
;
\end{eqnarray*}

\item 
\begin{equation*}
\left[ p_{1}\pi _{\Phi }\left( X\right) ^{\ast }p_{2}\left( p_{2}W_{\Phi
}\right) \mathcal{K}\right] =\left[ p_{1}\pi _{\Phi }\left( X\right) ^{\ast
}W_{\Phi }\mathcal{K}\right] =\left[ p_{1}\mathcal{H}_{\Phi }\right] =%
\mathcal{H}_{\Phi }\ominus \text{ker}\Delta _{1\Phi }\left( \Psi \right) ;
\end{equation*}

\item 
\begin{eqnarray*}
\Psi \left( x\right) &=&\Phi _{\sqrt{\Delta _{\Phi }\left( \Psi \right) }%
}\left( x\right) =W_{\Phi }^{\ast }\pi _{\Phi }\left( x\right) \sqrt{\Delta
_{1\Phi }\left( \Psi \right) }V_{\Phi } \\
&=&W_{\Phi }^{\ast }p_{2}\pi _{\Phi }\left( x\right) p_{1}\sqrt{\Delta
_{1\Phi }\left( \Psi \right) }V_{\Phi } \\
&=&\left( p_{2}W_{\Phi }\right) ^{\ast }p_{2}\pi _{\Phi }\left( x\right)
p_{1}\left( p_{1}\sqrt{\Delta _{1\Phi }\left( \Psi \right) }V_{\Phi }\right)
\end{eqnarray*}
\end{enumerate}

for all $x\in X$. Therefore, 
\begin{equation*}
\left( p_{2}\pi _{\Phi }\left( \cdot \right) p_{1},\mathcal{H}_{\Phi
}\ominus \ker \Delta _{1\Phi }\left( \Psi \right) ,\mathcal{K}_{\Phi
}\ominus \ker \Delta _{2\Phi }\left( \Psi \right) ,p_{1}\sqrt{\Delta _{1\Phi
}\left( \Psi \right) }V_{\Phi },p_{2}W_{\Phi }\right)
\end{equation*}%
is unitarily equivalent to the Stinespring construction associated to $\Psi
. $
\end{remark}

For $\Phi \in \mathcal{C}(X,L(\mathcal{H},\mathcal{K}))$, let $\widehat{\Phi 
}$ $=\{\Psi \in \mathcal{C}(X,L(\mathcal{H},\mathcal{K}));\Phi \backsim \Psi
\}$. Let $\Phi ,\Psi \in \mathcal{C}(X,L(\mathcal{H},\mathcal{K}))$. We say
that $\widehat{\Psi }\leq \widehat{\Phi }$ if $\ \Psi \curlyeqprec \Phi .$

For $\Phi \in \mathcal{C}(X,L(\mathcal{H},\mathcal{K}))$, let $[0,\widehat{%
\Phi }]=\{\widehat{\Psi };\Psi \in \mathcal{C}(X,L(\mathcal{H},\mathcal{K})),%
\widehat{\Psi }\leq \widehat{\Phi }\}$, and $[0,I]_{\Phi }=\{T\oplus S\in
\pi _{\Phi }\left( X\right) ^{\prime };0\leq T\oplus S\leq I\}.$

\begin{theorem}
Let $\Phi \in \mathcal{C}(X,L(\mathcal{H},\mathcal{K}))$. The map $[0,%
\widehat{\Phi }]\ni \widehat{\Psi }\rightarrow \Delta _{\Phi }\left( \Psi
\right) \in \lbrack 0,I]_{\Phi }$ is an order-preserving isomorphism.
\end{theorem}

\begin{proof}
By Theorem 2.12, the map $[0,\widehat{\Phi }]\ni \widehat{\Psi }\rightarrow
\Delta _{\Phi }\left( \Psi \right) \in \lbrack 0,I]_{\Phi }$ is well
defined. Let $\Psi _{1},\Psi _{2}\in \mathcal{C}(X,L(\mathcal{H},\mathcal{K}%
))$ such that $\Psi _{1}\curlyeqprec \Phi ,$ $\Psi _{2}\curlyeqprec \Phi $
and $\Delta _{\Phi }\left( \Psi _{1}\right) =\Delta _{\Phi }\left( \Psi
_{2}\right) $. Then $\Psi _{1}\backsim \Phi _{\sqrt{\Delta _{\Phi }\left(
\Psi _{1}\right) }}=\Phi _{\sqrt{\Delta _{\Phi }\left( \Psi _{2}\right) }%
}\backsim \Psi _{2},$ and so the map is injective.

Let $T\oplus S\in \lbrack 0,I]_{\Phi }$.$\ $Then $\Phi _{\sqrt{T\oplus S}%
}\in \mathcal{C}(X,L(\mathcal{H},\mathcal{K}))\ $and $\Phi _{\sqrt{T\oplus S}%
}$ $\curlyeqprec \Phi $, since, $T\in \pi _{\varphi }(A)^{\prime }$, $0\leq
T\leq I$, and by \cite[Theorem 1.4.2]{2}, $\Phi _{\sqrt{T\oplus S}}\left(
x\right) ^{\ast }\Phi _{\sqrt{T\oplus S}}\left( x\right) =\varphi _{T}\left(
\left\langle x,x\right\rangle \right) \leq $ $\varphi \left( \left\langle
x,x\right\rangle \right) =\Phi \left( x\right) ^{\ast }\Phi \left( x\right) $
for all $x\in X$. Moreover, since $\Delta _{\varphi }\left( \varphi
_{T}\right) =T$, by Remark 2.13(1), $\Delta _{\Phi }\left( \Phi _{\sqrt{%
T\oplus S}}\right) =T\oplus S$. Therefore, the map is surjective.

If $\widehat{\Psi _{1}}\leq \widehat{\Psi _{2}}$, then $\Psi
_{1}\curlyeqprec \Psi _{2}\curlyeqprec \Phi $, whence it follows that $\psi
_{1}\leq \psi _{2}\leq \varphi $ and, by \cite[Theorem 1.4.2]{2}, $\Delta
_{1\Phi }\left( \Psi _{1}\right) \leq \Delta _{1\Phi }\left( \Psi
_{2}\right) $. From this fact, and taking into account that $\pi _{\Phi }$
is nondegenerate, we deduce that $\Delta _{\Phi }\left( \Psi _{1}\right)
\leq \Delta _{\Phi }\left( \Psi _{2}\right) $ (see \cite[Note 4.6]{A}).
Conversely, if $0\leq T_{1}\oplus S_{1}\leq T_{2}\oplus S_{2}\leq I$, $%
T_{1}\oplus S_{1},T_{2}\oplus S_{2}\in \pi _{\Phi }\left( X\right) ^{\prime
} $, then $0\leq T_{1}\leq T_{2}\leq I$, $T_{1},T_{2}\in \pi _{\varphi
}\left( A\right) ^{\prime }$, whence it follows that $\varphi _{T_{1}}\leq $ 
$\varphi _{T_{2}}$, and so $\Phi _{\sqrt{T_{1}\oplus S_{1}}}\curlyeqprec
\Phi _{\sqrt{T_{1}\oplus S_{1}}}.$
\end{proof}

A completely positive map $\Phi $ on $X$ is \textit{pure} if for any
completely positive map $\Psi $ on $X$ with $\widehat{\Psi }\leq \widehat{%
\Phi }$, there is $\lambda >0$ such that $\Psi \backsim \lambda \Phi $.

\begin{proposition}
Let $\Phi $ be a non-zero element in $\mathcal{C}(X,L(\mathcal{H},\mathcal{K}%
))$. Then $\Phi $ is pure if and only if $\pi _{\Phi }$ is irreducible.
\end{proposition}

\begin{proof}
First, we suppose that $\Phi $ is pure. Let $T\oplus S\in \pi _{\Phi }\left(
X\right) ^{\prime }$ with $0\leq T\oplus S\leq I$. Then, by Theorem 2.15, $%
\Phi _{\sqrt{T\oplus S}}\curlyeqprec \Phi $, and since $\Phi $ is pure,
there is a positive number $\lambda $ such that $\Phi _{\sqrt{T\oplus S}%
}\backsim \lambda \Phi =\Phi _{\lambda I}$, whence $T\oplus S=\lambda ^{2}I$%
. Therefore, $\pi _{\Phi }\left( X\right) ^{\prime }=\mathbb{C}I$, and by 
\cite[Proposition 4.5]{A}, $\pi _{\Phi }$ is irreducible.

Conversely, let $\Psi \in \mathcal{C}(X,L(\mathcal{H},\mathcal{K}))$ such
that $\widehat{\Psi }\leq \widehat{\Phi }$. Then, by Theorem 2.15, there is $%
\lambda I\in \pi _{\Phi }\left( X\right) ^{\prime }=\mathbb{C}I$ with $%
\lambda >0$, such that $\Psi \backsim \Phi _{\sqrt{\lambda }I}=\sqrt{\lambda 
}\Phi $, and so $\Phi $ is pure.
\end{proof}

\end{document}